\documentclass[12pt]{article}
\usepackage{graphicx}
\usepackage{amssymb, amsmath, amsthm}

\newtheorem{thm}{Theorem}
\newtheorem{lem}[thm]{Lemma}
\newtheorem{cor}[thm]{Corollary}
\newtheorem{rem}{Remark}

\newtheorem{proposition}{Proposition}
\usepackage{mathptmx}
\usepackage{array}
\newcommand{\ord}{\operatorname{ord}}

\begin{document}

\title{Note about the linear complexity  of new  generalized cyclotomic  binary sequences of period
$2p^n$}
\author{Vladimir Edemskiy}

\maketitle

\begin{abstract}
This paper examines the linear complexity  of new generalized
cyclotomic binary sequences of period $2p^n$ recently proposed by Yi
Ouang et al. (arXiv:1808.08019v1 [cs.IT] 24 Aug 2018). We generalize
results obtained by them  and discuss author's conjecture of  this
paper.



\end{abstract}

\section{Introduction}
The cyclotomic classes and the generalized cyclotomic classes are
often used for design sequences with high linear complexity, which
is an important characteristic of sequence for the cryptography
applications \cite{CDR}. Recently, new generalized cyclotomic
classes were presented in \cite{ZCTY}. The linear complexity of new
generalized cyclotomic binary sequences with period $p^n$ was
studied in \cite{ZX,ELZH,YKW}. A new family of binary sequences with
period $2p^n$ based on the  generalized cyclotomic classes from
\cite{ZCTY} was presented in \cite{OX}. Yi Ouang et al. examined the
linear complexity of these sequences for $f=2^r$, where $p=1+ef$ and
$r$ is a positive integer. They offered new studying method of the
linear complexity of these sequences. Their  method based on ideas
from \cite{ELZH}.

In this paper we show that  for study of the  linear complexity of
new sequence family  from \cite{OX} we can use only old the method
from \cite{ELZH}. Furthermore, it will be enough for obtaining  more
generalized results than in \cite{OX} and for the proof and the
correction of the conjecture of the  authors of this paper. Here we
keep the notation and the structure of \cite{ELZH}, i.e., in Sect. 2
we introduce some basics and recall the definition of a generalized
cyclotomic sequence and the conjecture from \cite{OX}. Section 3 is
dedicated to the study of the linear complexity of this family of
cyclotomic sequences. Section 4 concludes the work in this paper.

\section{Preliminaries}

Throughout this paper, we will denote by $\mathbb{Z}_N$ the ring of
integers modulo $N$ for a positive integer $N$, and by
$\mathbb{Z}_N^{*}$ the multiplicative group of $\mathbb{Z}_N$.

\medskip

First of all we will recall some basics of the linear complexity of
a periodic sequence and introduce the generalized cyclotomic
sequences proposed in \cite{OX}.

\subsection{Linear Complexity}
Let $s^\infty=(s_0,s_1,s_2,\dots)$  be a binary sequence of period
$N$ and $S(x) = s_0 + s_1x +\cdots+ s_{N-1}x^{N-1}$. It is well
known (see, for instance, \cite[Page 171]{CDR}) that the linear
complexity of $s^\infty$ is given by
$$
L(s^\infty)=N-\deg \Big(\gcd \big(x^{N}-1,S(x)\big)\Big).
$$
So, if $N=2p^n$ then we see that
$$
L(s^\infty)=2p^n-\deg \Big(\gcd \big((x^{p^n}-1)^2,S(x)\big)\Big).
$$
Thus, if $\alpha_n$ is a primitive root of order $p^n$ of unity in
the extension of the field  $\mathbb{F}_2$ (the finite field of two
elements) then in order to find the linear complexity of a sequence
it is sufficient to find the zeros of $S(x)$ in the set
   $ \{ \alpha_n^i,  i=0,1,\ldots,p^n-1 \}$ and determine their multiplicity.

\subsection{New Generalized Cyclotomic Sequences Length $2p^n$}

\smallskip

Let $p$ be an odd prime and $p=ef+1$, where $e,f$ are positive
integers. Let $g$  be a primitive root modulo $p^n$. It is well
known \cite{IR} that an odd number from $g$ or $g+p^n$ is also a
primitive root modulo $2p^j$ for each integer $j\geq 1$. Hence, we
can assume that $g$ is an odd number. Further, the order of $g$
modulo $2p^j$ is equal to $\varphi(2p^j)=p^{j-1}(p-1)$, where
$\varphi(\cdot)$ is the Euler's totient function.   Below we recall
the definitions of generalized cyclotomic classes introduced in
\cite{ZCTY} and \cite{OX}.

\smallskip

Let $n$ be a positive integer. For $j=1, 2,\cdots, n$, denote
$d_j=p^{j-1}f$ and define
\begin{equation}\label{eq1}
\begin{split}
&D^{(p^j )}_0 = \left\{g^{t\cdot d_ j}~(\bmod{p^j})  \,|\,  0\leq t < e \right\},  \text{ and }  \\
&D^{(p^j )}_ i = g^i D^{(p^j )}_0 = \left\{g^i x~(\bmod{p^j}) : x
\in D^{(p^j )}_0\right\}, \quad 1\leq i <d_j,\\
&D^{(2p^j )}_0 = \left\{g^{t\cdot d_ j}~(\bmod{2p^j})  \,|\,  0\leq t < e \right\},  \text{ and }  \\
&D^{(2p^j )}_ i = g^i D^{(p^j )}_0 = \left\{g^i x~(\bmod{2p^j}) : x
\in D^{(2p^j )}_0\right\}, \quad 1\leq i <d_j.
\end{split}
\end{equation}
The cosets $D^{(p^j )}_ i$, $i=0, 1, \cdots, d_j-1$, are called
\textit{generalized cyclotomic classes} of order $d_ j$ with respect
to $p^j$. It was shown in \cite{ZCTY} that $\left\{D^{(p^j )}_ 0,
D^{(p^j )}_ 1,\dots,D^{(p^j)}_ {d_j-1}\right\}$ forms a partition of
$\mathbb{Z}^{*}_{p^j}$ for each integer $j \geq 1$ and for an
integer $m\geq 1$,
$$ \mathbb{Z}_{p^m}=\bigcup_{j=1}^{m} \bigcup_{i=0}^{d_j-1}
p^{m-j}D^{(p^j )}_ i \cup \{0\}.$$

Also $\left\{D^{(2p^j )}_ 0, D^{(2p^j )}_ 1,\dots,D^{(2p^j)}_
{d_j-1}\right\}$ forms a partition of $\mathbb{Z}^{*}_{2p^j}$ for
each integer $j \geq 1$ and for an integer $m\geq 1$,
$$ \mathbb{Z}_{2p^m}=\bigcup_{j=1}^{m} p^{m-j}\Bigl (\bigcup_{i=0}^{d_j-1}
D^{(p^j )}_ i \cup 2D^{(p^j )}_ i \Bigr )  \cup \{0\} \cup
\{p^m\}.$$
\medskip

Let $f$ be a positive even integer and $b$ an integer with $0\leq b
< p^{n-1} f$. Define four sets
\begin{multline*}
 \mathcal{C}_0^{(2p^n)}=\bigcup_{j=1}^{n} \bigcup_{i=d_j/2}^{d_j-1}
p^{n-j}\Bigl(D^{(2p^j )}_ {(i+b)\pmod{d_j}}\cup 2D^{(2p^j )}_
{(i+b)\pmod{d_j}}\Bigr) \cup\{p^n\} , \text{
    and  }\\
\mathcal{C}_1^{(2p^n)}=\bigcup_{j=1}^{n} \bigcup_{i=0}^{d_j/2-1}
p^{n-j}\Bigl(D^{(2p^j )}_ {(i+b)\pmod {d_j}}\cup 2D^{(2p^j )}_
{(i+b)\pmod{d_j}}\Bigr) \cup \{0\},\\
\mathcal{\widetilde{C}}_0^{(2p^n)}=\bigcup_{j=1}^{n}
p^{n-j}\Bigl(\bigcup_{i=0}^{d_j/2-1} 2 D^{(2p^j )}_ {(i+b)\pmod
{d_j}} \cup \bigcup_{i=d_j/2}^{d_j-1}D^{(2p^j )}_
{(i+b)\pmod{d_j}}\Bigr) \cup\{p^n\} , \text{
    and  }
    \end{multline*}
    \begin{multline}\label{eq2}
\mathcal{\widetilde{C}}_1^{(2p^n)}=\bigcup_{j=1}^{n}
p^{n-j}\Bigl(\bigcup_{i=0}^{d_j/2-1} D^{(2p^j )}_ {(i+b)\pmod
{d_j}}\cup \bigcup_{i=d_j/2}^{d_j-1} 2D^{(2p^j )}_
{(i+b)\pmod{d_j}}\Bigr) \cup \{0\}.
\end{multline}
 It is obvious that $\mathbb{Z}_{2p^n}=\mathcal{C}_0^{(2p^n)} \cup
\mathcal{C}_1^{(2p^n)}=\mathcal{\widetilde{C}}_0^{(2p^n)} \cup
\mathcal{\widetilde{C}}_1^{(2p^n)}$ and
$|\mathcal{C}_i^{(2p^n)}|=|\mathcal{\tilde{C}}_i^{(2p^n)}|=p^n, \;\
i=0,1$. Families of  balanced binary sequences
$s^\infty=(s_0,s_1,s_2,\dots)$ and
$\tilde{s}^\infty=(\tilde{s}_0,\tilde{s}_1,\tilde{s}_2,\dots)$ of
period $p^n$ can thus be defined as in \cite{OX}, i.e.,
\begin{equation}
\label{eq3} s_i =\begin{cases}
 0,&\text{ if  }   i~(\bmod~p^n)  \in \mathcal{C}_0^{(2p^n)}, \\
  1,&\text{ if }   i~(\bmod~p^n)  \in \mathcal{C}_1^{(2p^n)}. \\
 \end{cases}
\end{equation}
and
\begin{equation}
\label{eq4} \tilde{s}_i =\begin{cases}
 0,&\text{ if  }   i~(\bmod~p^n)  \in \mathcal{\widetilde{C}}_0^{(2p^n)}, \\
  1,&\text{ if }   i~(\bmod~p^n)  \in \mathcal{\widetilde{C}}_1^{(2p^n)}. \\
 \end{cases}
\end{equation}
In the case of $f=2^r$, the linear complexity of $s^\infty,
\tilde{s}^\infty$ was  estimated  in \cite{OX}, where a conjecture
about the linear complexity of these sequences was also made as
follows.

 \textbf{Conjecture.}
 (1) If $2^e \equiv -1 \pmod {p}$ but $2^e \not \equiv -1 \pmod {p^2}$, then the linear complexity $L(s^\infty) = 2p^n- (p - 1).$

  (2) If $2^e \equiv 1 \pmod {p}$ but $2^e \not \equiv 1 \pmod {p^2}$, then the linear complexity $L(\tilde{s}^\infty) = 2p^n- (p -
  1)- e$.

\smallskip

\subsection{Main Result}

This subsection will study the linear complexity of $s^{\infty}$,
$\tilde{s}^{\infty}$  in \eqref{eq3} and \eqref{eq4} for some even
integers $f$ and when $p$ is not a Wieferich prime, i.e.
$2^{p-1}\not\equiv 1 \pmod{p^2}$. It was shown that there are only
two such primes, 1093 and 3511, up to $6 \times10^{17}$
\cite{AS,DK}. The main result in this paper is given as follows.

\begin{thm}
    \label{t1} Let $p=ef+1$ be an odd prime with $2^{p-1}\not\equiv 1 \pmod{p^2}$ and $f$ is an even positive integer.
     Let $\ord_p(2)$ denote the order of $2$ modulo $p$ and
    $v=\gcd\big(\frac{p-1}{{\rm
            ord}_p(2)}, f\big)$.

  \noindent(i) Let $s^\infty$ be a
    generalized cyclotomic binary sequence of period $p^n$ defined in
    \eqref{eq3}. Then the linear complexity of $s^{\infty}$ is given by
    $$
    \begin{array}{c}
   L(s^\infty) = 2p^n -r\cdot {\rm ord}_p(2), \quad 0\leq r \leq \frac{p-1}{{\rm
ord}_p(2)}.
    \end{array}
    $$

    Furthermore,
     the linear complexity
    $$
   L(s^\infty)= \begin{cases}
    2p^n-p+1, & \text{ if } v = f/2 ; \\
    2p^n, & \text{ if } v=1 \text{ or } 2v | \frac{f}{2}, \text{ or } f=v.
    \end{cases}
    $$

\noindent(ii) Let $\tilde{s}^\infty$ be a
    generalized cyclotomic binary sequence of period $p^n$ defined in
    \eqref{eq4}. Then for the linear complexity of $\widetilde{s}^{\infty}$
    we have
    $$
    \begin{array}{c}
2p^n -2 r\cdot {\rm ord}_p(2)\leq L(\widetilde{s}^\infty) \leq
2p^n-r\cdot {\rm ord}_p(2), \quad 0\leq r \leq \frac{p-1}{{\rm
ord}_p(2)}.
    \end{array}
    $$

    Furthermore,
      the linear complexity
    $$
   L(\tilde{s}^\infty) = \begin{cases}
    2p^n-3(p-1)/2 & \text{ if } v = f ; \\
    2p^n, & \text{ if } v | \frac{f}{2}, \text{ or } v=2, v\neq f.
    \end{cases}
    $$

\end{thm}
\begin{cor} Let $f=2^r$. Then:

\noindent(i) The linear complexity of $s^{\infty}$ is given by
$$
    L(s^\infty) = \begin{cases}
    2p^n-p+1, & \text{ if } v = f/2 ; \\
    2p^n, & \text{ otherwise }.
    \end{cases}
    $$

\noindent(ii) The linear complexity of $\widetilde{s}^{\infty}$ is
given by
    $$
    L(\tilde{s}^\infty) = \begin{cases}
    2p^n-3(p-1)/2, & \text{ if } v = f ; \\
     2p^n, & \text{ otherwise }.
    \end{cases}
    $$
    \end{cor}

\begin{rem}Suppose $2\equiv g^u \pmod{p}$ for some integer $u$. It is easily seen that $\gcd\big(\frac{p-1}{{\rm
            ord}_p(2)}, f\big)=\gcd(u, f)$. Thus the condition $2^e\equiv 1\pmod {p} $ in Conjecture from \cite{OX} is equivalent to $v=\gcd\big(\frac{p-1}{{\rm
            ord}_p(2)}, f\big) = f$ and the condition $2^e\equiv -1\pmod {p}
            $ is equivalent to $v=f/2$.
    In the case that $f=2^r$ for a positive integer $r$,  the integer $v$ is also a power of $2$, which either equals $f$ or $f/2$ or divides $f/4$.
     Hence Conjecture from \cite{OX} is included in Theorem \ref{t1} as a special
     case. Here we make the correction of Conjecture (ii).
\end{rem}

If $2$ is a primitive roots modulo $p$ then $v=1$.

\smallskip

For the proof of Theorem \ref{t1} we will use the same definitions
and same method that as \cite{ELZH}.

Let $S(x)=s_0+s_1x+\cdots+s_{2p^n-1}x^{2p^n-1}$ and
$\widetilde{S}(x)=\widetilde{s}_0+\widetilde{s}_1x+\cdots+\widetilde{s}_{2p^n-1}x^{2p^n-1}$
for the generalized cyclotomic sequences $s^{\infty}$,
$\widetilde{s}^{\infty}$ defined in \eqref{eq3} and  \eqref{eq4},
respectivly. Then,
\begin{multline}\label{eq5}
S(x) =\sum\limits_{t\in \mathcal{C}_1^{(p^n)}} x^t = 1 +
\sum\limits_{j=1}^n\sum\limits_{i=0}^{d_j/2-1}\Bigl(\sum_{t \in
D^{(2p^j)}_{i+b \pmod{d_j}}} x^{p^{n-j}t}+\sum_{t \in
2D^{(2p^j)}_{i+b \pmod{d_j}}} x^{p^{n-j}t}\Bigr),  \text{and}
\\
\widetilde{S}(x) = \sum\limits_{t\in
\mathcal{\widetilde{C}}_1^{(p^n)}} x^t = 1
+\sum\limits_{j=1}^n\sum\limits_{i=0}^{d_j/2-1} \sum_{t \in
D^{(2p^j)}_{i+b \pmod{d_j}}} x^{p^{n-j}t}+
\sum\limits_{j=1}^n\sum\limits_{i=d_j/2}^{d_j}\sum_{t \in
2D^{(2p^j)}_{i+b \pmod{d_j}}} x^{p^{n-j}t}.
\end{multline}

For simplicity of presentation, we define polynomials as in
\cite{ELZH}
\begin{equation}\label{eq6}
E^{(p^j)}_i(x) =\sum_{t \in D^{(p^j)}_i} x^t, \quad 1\leq j \leq n,
\, 0\leq i<d_j,
\end{equation}
and
\begin{equation}\label{eq7}
\begin{split}
H^{(p^j)}_{k}(x) & =\sum_{i=0}^{ d_j/2-1}
E^{(p^j)}_{i+k\pmod{d_j}} (x),  \quad 0\leq k < d_j,\\
T^{(p^m)}_{k }(x)&=\sum\limits_{j=1}^{m}H^{(p^j)}_{k}(x^{p^{m-j}}),
\quad m=1, 2, \cdots, n.
\end{split}
\end{equation}
Notice that the subscripts $i$ in $D^{(p^j)}_i$, $H^{(p^j)}_i(x)$
and $T^{(p^j)}_i(x)$ are all taken modulo the order $d_j$. In the
rest of this paper the modulo operation will be omitted when no
confusion can arise.

\medskip

Let $\overline{\mathbb{F}}_2$ be an algebraic closure of
$\mathbb{F}_2$ and $\alpha_n \in \overline{\mathbb{F}}_2$  be a
primitive $p^n$-th root of unity. Denote
$\alpha_j=\alpha_n^{p^{n-j}},\, j=1,2\dots,n-1$.

The properties of considered polynomials were studied in
\cite{ELZH}. We have here the following statement.
\begin{lem}  \label{l1} \cite{ELZH}
    For any $a\in D^{(p^j)}_k$,  we have

\noindent(i) $T^{(p^m)}_i(\alpha_m^{p^la})
=T^{(p^{m-l})}_{i+k}(\alpha_{m-l})  + (p^l-1)/2 \pmod{2}$ for $0\leq
l <m$;  and

\noindent(ii) $T^{(p^m)}_i(\alpha_m^{a})  +
T^{(p^m)}_{i+d_m/2}(\alpha_m^{a})  = 1$.

\noindent(iii) Let $p$ be a non-Wieferich prime. Then
$T^{(p^m)}_i(\alpha_m) \not \in \{0,1\}$ for $m>1$.

 \noindent(iv)
Let $p$ be a non-Wieferich prime. Then $T^{(p^m)}_i(\alpha_m)  +
T^{(p^m)}_{i+f/2}(\alpha_m)  \neq 1$ for $m>1$.
\end{lem}
Throughout this paper an integer $u$ will be such that  $2\equiv g^u
\pmod{p^n}$. Now we will show that the studying of linear complexity
of above sequences is equivalent to the  investigation of properties
of $T^{(p^m)}_i(x)$

\begin{proposition}
\label{prop1}   Let $\alpha_n $ be a $p^n$-th primitive root of
unity and let $2\equiv g^u \pmod{p^n}$. Given any element $a\in
\mathbb{Z}_{p^n}$, we have

\noindent(i)
$S(\alpha_n^{a})=1+T^{(p^{n})}_{b}(\alpha_{n}^a)+T^{(p^{n})}_{b+u}(\alpha_{n}^a)$;
and

\noindent(ii)
$S(\alpha_n^{a})=T^{(p^{n})}_{b}(\alpha_{n}^a)+T^{(p^{n})}_{b+u}(\alpha_{n}^a)$.

\end{proposition}
\begin{proof}
(i) Since $ \sum_{t\in  p^{n-j}D^{(2p^j )}_
{(i+b)}}\alpha^{at}=\sum_{t\in  p^{n-j}D^{(p^j )}_
{(i+b)}}\alpha^{at}$ by \eqref{eq1}, it follows from  our
definitions and Lemma \ref{l1} that
\begin{multline*}S(\alpha_n^{a})=\sum\limits_{t\in
\mathcal{C}_1^{(2p^n)}}\alpha^{at}=\sum_{j=1}^{n}
\sum_{i=0}^{d_j/2-1} \bigl(\sum_{t\in  p^{n-j}D^{(2p^j )}_
{(i+b)}}\alpha^{at} +\sum_{t\in 2p^{n-j}D^{(2p^j )}_
{(i+b)}}\alpha^{at} \bigr) +1=\\
1+T^{(p^{n})}_{b}(\alpha_{n}^a)+T^{(p^{n})}_{b}(\alpha_{n}^{2a})=1+T^{(p^{n})}_{b}(\alpha_{n}^a)+T^{(p^{n})}_{b+u}(\alpha_{n}^{a}).
\end{multline*}

(ii) Similarly we have
$$\widetilde{S}(\alpha_n^{a})=1+T^{(p^{n})}_{b}(\alpha_{n}^a)+T^{(p^{n})}_{b+u+d_n/2}(\alpha_{n}^a)=T^{(p^{n})}_{b}(\alpha_{n}^a)+T^{(p^{n})}_{b+u}(\alpha_{n}^a).$$
\end{proof}

We now examine the value of
$T_b^{(p^n)}(\alpha_{n}^i)+T^{(p^{n})}_{b+u}(\alpha_{n}^i)$ for some
integers $i\in \mathbb{Z}_{p^n}$.
\begin{proposition}
\label{prop2} Let $p$ be a non-Wieferich prime. Then
$S(\alpha_n^{i})\neq 0$ and $\widetilde{S}(\alpha_n^{a})\neq 0$ for
 $i \in
\mathbb{Z}_{p^n}\setminus p^{n-1}\mathbb{Z}_{p}$.
\end{proposition}
\begin{proof}
    This is sufficient to prove that $T^{(p^n)}_b(\alpha_n^i)+T^{(p^{n})}_{b+u}(\alpha_{n}^i) \not \in
\{0,1\}$ for
 $i \in
\mathbb{Z}_{p^n}\setminus p^{n-1}\mathbb{Z}_{p}$ and $b=0, 1,
\cdots, d_n-1.$ As it was shown in \cite{ELZH} that without loss of
generality it is enough proof,
$T^{(p^m)}_0(\alpha_m)+T^{(p^m)}_{u}(\alpha_m)\not \in \{0,1\}$ for
$m>1$.

We consider two
cases.

1. Let $T^{(p^m)}_0(\alpha_m)+T^{(p^m)}_{u}(\alpha_m)=0$. Since
$(T^{(p^m)}_0(\alpha_m))^2=T^{(p^m)}_{u}(\alpha_m)=0$, we see that
in this case $T^{(p^m)}_0(\alpha_m) \in \{0,1\}$. We obtain a
contradiction with Lemma \ref{l1} (iii).

2. Let $T^{(p^m)}_0(\alpha_m)+T^{(p^m)}_{u}(\alpha_m)=1$.
\smallskip

It then follows from Lemma \ref{l1} (i) that
$$1=\left (T^{(p^m)}_0(\alpha_m)+T^{(p^m)}_{u}(\alpha_m)\right )^2=  T^{(p^m)}_0(\alpha^2_m)+T^{(p^m)}_{u}(\alpha_m^2) =
T^{(p^m)}_{u}(\alpha_m)+T^{(p^m)}_{2u}(\alpha_m),$$ which implies $
T^{(p^m)}_{iu}(\alpha_m) + T^{(p^m)}_{(i+1) u}(\alpha_m) = 1 $ for
any integer $i\geq 1$. Hence $ T^{(p^m)}_{0}(\alpha_m) =
T^{(p^m)}_{2i u}(\alpha_m).$
\smallskip

Denote $w=\gcd(2u,d_m)$. Since $p$ is a non-Wieferich prime, it
follows by \cite{ELZH} that $w$ divides $f$. Since the subscript of
$T_i^{(p^m)}(x)$ is taken modulo $d_m$, it is easily seen that
\begin{equation}
\label{eq8}  T^{(p^m)}_{0}(\alpha_m) =
T^{(p^m)}_{iw}(\alpha_m),\quad \text{ for any integer }i\geq 1.
\end{equation}
By  Lemma  \ref{l1} (ii) from the last formula we have
$T^{(p^m)}_{d_m/2}(\alpha_m)=T^{(p^m)}_{d_m/2+iw}(\alpha_m)$ or
$T^{(p^m)}_{d_m}(\alpha_m)=T^{(p^m)}_{d_m/2+jf}(\alpha_m)$. Then we
get that
$$  T^{(p^m)}_{d_m/2+(p^{m-1}+1)/2\cdot f}(\alpha_m) =
T^{(p^m)}_{p^{m-1}f/2+(p^{m-1}+1)f/2}(\alpha_m)=
T^{(p^m)}_{f/2+d_m}(\alpha_m)= T^{(p^m)}_{f/2}(\alpha_m).$$
 Hence,
$T^{(p^m)}_{d_m/2}(\alpha_m)=T^{(p^m)}_{f/2}(\alpha_m)$. Thus, by
Lemma  \ref{l1} (ii) we obtain that
$T^{(p^m)}_{0}(\alpha_m)+1=T^{(p^m)}_{f/2}(\alpha_m)$. But the
latest equality is not possible for $m>1$ by Lemma \ref{l1} (iv).
\end{proof}
\smallskip

By Proposition \ref{prop2}, we only need to study the value of
$T_b^{(p^n)}(\alpha_n^{i})+T_{b+u}^{(p^n)}(\alpha_n^{i})$ for
integers $i$ in the set $p^{n-1}\mathbb{Z}_{p}$. Suppose $i=p^{n-1}
a,\;\  a\in D_{i}^{(p)}$. Then, it follows from Proposition
\ref{prop1} and Lemma \ref{l1} that
$$S(\alpha_n^{i})=1+T^{(p^n)}_b(\alpha_n^i)
+T^{(p^n)}_{b+u}(\alpha_n^i)=1+T^{(p)}_b(\alpha_1^a)
+T^{(p)}_{b+u}(\alpha_1^a)=1+
H^{(p)}_k(\alpha_1)+H^{(p)}_{k+u}(\alpha_1),$$ where  $k \equiv b +
i \pmod {f}$. The following proposition examines the value of
$H_k^{(p)}(\alpha_1)+H_{k+u}^{(p)}(\alpha_1)$ according to the
relation between $f$ and $\ord_p(2)$.

\begin{proposition}
    \label{l7} Let $p = ef + 1$ be an odd prime with $f$ being an even  positive integer and $v = \gcd(\frac{p-1}{\ord_p(2)}, f) $.
    Then,

\noindent(i) $\left| \Big\{k \in \mathbb{ Z}_{f}
\,|\,H_k^{(p)}(\alpha_1)+H_{k+u}^{(p)}(\alpha_1) =  0\Big\}\right| =
\begin{cases}
f,&  \text{if  } v = f,\\
0,&  \text{if } v | f/2 \text{ or } v=2, v\neq f.
\end{cases}$

 \noindent(ii)$\left| \Big\{k \in \mathbb{ Z}_{f}
\,|\,H_k^{(p)}(\alpha_1)+H_{k+u}^{(p)}(\alpha_1) =  1\Big\}\right| =
\begin{cases}
f,&  \text{if  } v = f/2,\\
0,&  \text{if  } v=1, \text{ or } v=f \text{ or } 2v|f/2.
\end{cases}$
\end{proposition}

\begin{proof}
        Since $\ord_p(2)  = \frac{p-1}{\gcd(p-1, \,u)}$,  it follows that $\gcd(u, f) = \gcd(\frac{(p-1)}{\ord_p(2)}, f) = v$ \cite{ELZH}.

    (i) For $v=f$ this statement is clear.

    Let $v|f/2$ or $v=2, v\neq f$. We shall prove this case by contradiction. Suppose $H_k^{(p)}(\alpha_1)+H_{k+u}^{(p)}(\alpha_1)=0$ for some integer $k$.
   Since  $(H_k^{(p)}(\alpha_1))^2 = H_{k+u}^{(p)}(\alpha_1)$, it
   follows that $H_k^{(p)}(\alpha_1)) \in \{0,1\}$. By \cite{ELZH}
   this is not possible for $v|f/2$ or $v=2, v\neq f$.

 (ii) For $v=f/2$ this statement is clear. If $v=f$ then $2\in
 D^{(p)}_0$ and we have
 $H_k^{(p)}(\alpha_1)+H_{k}^{(p)}(\alpha_1)=1$. This is
 impossible

    Suppose $H_k^{(p)}(\alpha_1)+H_{k+u}^{(p)}(\alpha_1)=1$ for some integer $k$. Without loss of generality,
    we assume $k=0$ and $H_0^{(p)}(\alpha_1) = H_{u}^{(p)}(\alpha_1)+1$.

    In the case when $v \neq f$. Since $\gcd(u, f) = \gcd(\frac{(p-1)}{\ord_p(2)}, f) = v$, by a similar argument as in the proof of Proposition \ref{prop2} we get
    $$
    H_0^{(p)}(\alpha_1) =   H_{2v}^{(p)}(\alpha_1) = \cdots = H_{2vi}^{(p)}(\alpha_1).
    $$
   So, if  $2v$ divides $f/2$, then $H_{f/2}^{(p)}(\alpha_1) = H_{{ 2v \cdot f/4v}}^{(p)}(\alpha_1) = H_0^{(p)}(\alpha_1),$ which is a contradiction.

 Let $v=1$.       Then we get
 $H_{i}^{(p)}(\alpha_1)+ H_{i+1}^{(p)}(\alpha_1)+1=0,
i=0,1,\dots,f-1$ and then $ E_{i}^{(p)}(\alpha_1)+
E_{i+f/2}^{(p)}(\alpha_1)+1=0, \quad i=0, 1,\dots,f-1. $ In
\cite{ELZH} it was shown  that this is  impossible.

\end{proof}

\noindent\textbf{Proof of Theorem \ref{t1}.} Recall that the linear
complexity of $s^{\infty}$ is given by
 $$
L(s^{\infty})=N-\deg \Big(\gcd \big((x^{p^n}-1)^2,S(x)\big)\Big).
$$

\noindent (i) From Proposition \ref{prop2} we know
$S(\alpha_n^{i})\neq 0$ for
 $i \in
\mathbb{Z}_{p^n}\setminus p^{n-1}\mathbb{Z}_{p}$. For the remaining
set $p^{n-1}\mathbb{ Z}_p$, if $i=0$, then $S(1)=1$; if $i \in
p^{n-1}\mathbb{ Z}_p^*$, we have
$$
S(\alpha_n^{i})=1+T^{(p^n)}_b(\alpha_n^i)
+T^{(p^n)}_{b+u}(\alpha_n^i)=1+T^{(p)}_b(\alpha_1^a)
+T^{(p)}_{b+u}(\alpha_1^a)=1+
H^{(p)}_b(\alpha_1^a)+H^{(p)}_{b+u}(\alpha_1^a)
$$ for some integer $a\in \mathbb{ Z}_p^*$.

Suppose $H^{(p)}_k(\alpha_1^a)+H^{(p)}_{k+u}(\alpha_1^a)=1$ for some
integer $k$. Then
$$
1 = (H^{(p)}_k(\alpha_1))^{2} + H^{(p)}_{k+u}(\alpha_1^{2}) =
H^{(p)}_{k+u}(\alpha_1)+H^{(p)}_{k+2u}(\alpha_1),
$$ and so on (here $u\not \equiv 0 \pmod{f})$).
So, we have
$$
|\{i:\;\ S(\alpha_n^{i})=0, i=1,2,\dots,p^n-1\}|= r\ord_p(2).
$$
where $r$ is an integer with $0\leq r \leq \frac{p-1}{\ord_p(2)}$.

Further, by \eqref{eq5} we see that
$$xS'(x)=\sum\limits_{j=1}^n\sum\limits_{i=0}^{d_j/2-1} \sum_{t
\in D^{(2p^j)}_{i+b \pmod{d_j}}} x^{p^{n-j}t}.$$

Hence, $\alpha_n^{i} S(\alpha_n^{i})=T^{(p^n)}_b(\alpha_n^i)$. So,
if $\alpha_n^{i}$ is a root of $S(x)$ and $S'(x)$ then $1+
T^{(p^n)}_b(\alpha_n^i)+ (T^{(p^n)}_{b}(\alpha_n^i))^2=0$ and
$T^{(p^n)}_b(\alpha_n^i)=0$. It is not possible and any root of
$S(x)$ is simple.

Then the statement of this theorem follows from Proposition
\ref{prop2}.

\smallskip

\noindent (ii) In this case
$$
S(\alpha_n^{i})= H^{(p)}_b(\alpha_1^a)+H^{(p)}_{b+u}(\alpha_1^a)
$$ for some integer $a\in \mathbb{ Z}_p^*$.

Then as earlier we again get
$$
|\{i:\;\ S(\alpha_n^{i})=0, i=1,2,\dots,p^n-1\}|= r\ord_p(2).
$$
where $r$ is an integer such that  $0\leq r \leq
\frac{p-1}{\ord_p(2)}$.

Here, by \eqref{eq5} we see that
$$x\widetilde{S}'(x)=\sum\limits_{j=1}^n\sum\limits_{i=0}^{d_j/2-1}\sum_{t
\in D^{(2p^j)}_{i+b \pmod{d_j}}} x^{p^{n-j}t}.$$
and also
$\alpha_n^{i} \widetilde{S}(\alpha_n^{i})=T^{(p^n)}_b(\alpha_n^i)$.
If $v = f$ then it follows from \cite{ELZH} that
$$ |\{i:\;\ T^{(p^n)}_b(\alpha_n^i)=0, i=1,2,\dots,p^n-1\}|= (p-1)/2.
$$
Then the statement of this theorem follows from Proposition
\ref{prop2}.

\smallskip

 \hfill$\square$

\subsection{Additional remark}
Let $p$ be  a Wieferich prime.  Wieferich primes are very rare
\cite{DK}, hence we could ignore these numbers but nonetheless we
show that the old method also works in this case.  In this
subsection we consider only the case when $f=2^r$, where $r$ is a
positive integer. Denote $D=\{k: \;\ 2^{p-1} \equiv 1 \pmod{p^k}\}$
and $wn=\max\limits_{k \in D} {k}$.

Suppose $2\equiv g^u \pmod {p^{nw}}$. Then $u \equiv 0 \pmod
{p^{nw-1}}$. Thus, $u=p^{nw-1} z$ where $\gcd(z,p)=1$. It is easy to
check that  $2\equiv g^{p^{j-1}z} \pmod {p^j}$ for $j\leq nw$.

Let $v=\gcd(z,f)$. First, we  study the value of
$T_k^{(p^j)}(\alpha_j^{i})$ for integers $i$ in the set
$\mathbb{Z}_{p^j}$. Let $T_k^{(p^j)}(\alpha_j^{i}) \in \{0,1\}$ and
$v\neq f$. Without loss of generality,
    we assume  $T_k^{(p^j)}(\alpha_j^{i})=0$. As earlier we
    obtain that $$0=T_k^{(p^j)}(\alpha_j^{i})=T_{k+lvp^{j-1}}^{(p^j)}(\alpha_j^{i}) \text{ for } l=0,1,2,\dots.$$
    Since $ vp^{j-1}$ divides $d_j/2=p^{j-1}f/2$ for $f=2^r$, we
    have a contradiction. So, $T_k^{(p^j)}(\alpha_j^{i}) \in
    \{0,1\}$ for $i\in \mathbb{Z}_{p^j}$ only when $v=f$

We consider a few cases.

(i) Suppose $v=f$. Then $2\in D_0^{(p^j)}$ for $j\leq nw$. In this
case $T^{(p^j)}_k(\alpha_j^i)=\bigl(T^{(p^j)}_k(\alpha_j^i)\bigr)^2$
and $T^{(p^j)}_k(\alpha_j^i)\in \{0,1\}$ for any $k$ and $i\in
\mathbb{Z}_{p^j}$. Thus, $S(\alpha_j^i)=1$ for $i\in
\mathbb{Z}_{p^j}$. Further, $\widetilde{S}(\alpha_j^i)= 0$ for $i\in
\mathbb{Z}_{p^j}, \; \ i\neq 0$ and $|\{i:
\widetilde{S}'(\alpha_j^i)=0, i=1,\dots, p^j-1\}|=(p^j-1)/2$.

(ii) Suppose $v=f/2$. In this case
$\bigl(T^{(p^j)}_k(\alpha_j^i)\bigr)^2=T^{(p^j)}_{k+d_j/2}(\alpha_j^i)$.
 Thus, by Lemma \ref{l1} $\widetilde{S}(\alpha_j^i)=1$ for $i\in
\mathbb{Z}_{p^j}$. Further, $S(\alpha_j^i)= 0$ and
$S'(\alpha_j^i)\neq 0$ for $i\in \mathbb{Z}_{p^j}, \; i\neq 0$.

(iii) $v\neq f/2,f$. Here
$\bigl(T^{(p^j)}_k(\alpha_j^i)\bigr)^2=T^{(p^j)}_{k+vp^{j-1}}(\alpha_j^i)$.
So, if  $T^{(p^j)}_k(\alpha_j^i)+T^{(p^j)}_{k+u}(\alpha_j^i)=0$ then
$T^{(p^j)}_k(\alpha_j^i)=T^{(p^j)}_{k+lvp^{j-1}}(\alpha_j^i)$ for
$l\geq 0$. Also, if
$T^{(p^j)}_k(\alpha_j^i)+T^{(p^j)}_{k+u}(\alpha_j^i)=1$ then
$T^{(p^j)}_k(\alpha_j^i)=T^{(p^j)}_{k+2lvp^{j-1}}(\alpha_j^i)$ for
$l\geq 0$.

Since $f=2^r$ and $v\neq f/2, f$, it follows that $2vp^{j-1}$
divides $p^{j-1}f/2$. We obtain a contradiction with Lemma \ref{l1}.

If $j\geq wn$ then $[\mathbb{F}_2(\alpha_{j+1}):
\mathbb{F}_2(\alpha_j)]=p,$ where $\alpha_j=\alpha_n^{p^{n-j}}$ and
$\alpha_n$ is a primitive $p^n$-th root of unity. In this case we
can use the method from \cite{ELZH} as  earlier.

Let $m=\min(n,wn)$.  So, for $f=2^r$ we can obtain  that
 the linear complexity of $s^{\infty}$  is given by

    $$
    L(s^{\infty}) = \begin{cases}
    2p^n-(p^{m}-1), & \text{ if } v = f/2 ; \\
    2p^n, & \text{ otherwise },
    \end{cases}
    $$
and the linear complexity of $\widetilde{s}^{\infty}$ for $n\geq wn$
is given by

    $$
    L(\tilde{s}^{\infty}) = \begin{cases}
    2p^n-3(p^{m}-1)/2, & \text{ if } v = f ; \\
    2p^n, & \text{ otherwise }.
    \end{cases}
    $$


\begin{thebibliography}{99}
    \providecommand{\url}[1]{{#1}}
    \providecommand{\urlprefix}{URL }
    \expandafter\ifx\csname urlstyle\endcsname\relax
    \providecommand{\doi}[1]{DOI~\discretionary{}{}{}#1}\else
    \providecommand{\doi}{DOI~\discretionary{}{}{}\begingroup
        \urlstyle{rm}\Url}\fi
 \bibitem{AS}
 Akbary, A., Siavashi S.: The largest known Wieferich numbers,
Integers \textbf{18-A3}  1-6 (2018).
       \bibitem{CDR}
    Cusick, T., Ding, C., Renvall, A.: Stream Ciphers and Number Theory.
    \newblock North-Holland mathematical library. Elsevier (2004).


    \bibitem{DK}
    Dorais, F.G., Klyve, D.: A {Wieferich} prime search up to $6.7 \times 10^{15}$.
    \newblock Journal of Integer Sequences \textbf{14}(11.9.2), 1--14 (2011).
    \newblock

\bibitem{ELZH}
Edemskiy, V., Li, C., Zeng, X., Helleseth, T.: The linear complexity
of generalized cyclotomic binary sequences of period $p^n$. Designs,
Codes and Cryptography. PP., 1-15. //DOI: 10.1007/s10623-018-0513-2


    \bibitem{IR}
    Ireland, K., Rosen, M.: A Classical Introduction to Modern Number Theory.
    \newblock Graduate Texts in Mathematics. Springer (1990).

\bibitem{OX}
 Ouyang, Y.,  Xie, X,: Linear complexity of generalized cyclotomic
sequences of period $2p^m$. arXiv:1808.08019v1  [cs.IT]  24 Aug 2018


    \bibitem{YKW}
    Ye, Z., Ke, P., Wu, C.:  A further study of the linear complexity of new binary cyclotomic sequence of length $p^n$.
    AAECC (2018). https://doi.org/10.1007/s00200-018-0368-9

    \bibitem{ZCTY}
    Zeng, X., Cai, H., Tang, X., Yang, Y.: Optimal frequency hopping sequences of
    odd length.
    \newblock IEEE Transactions on Information Theory \textbf{59}(5), 3237--3248
    (2013).
 \bibitem{ZX}
    Xiao, Z., Zeng, X., Li, C., Helleseth, T.: New generalized cyclotomic binary
    sequences of period $p^2$.
    Des. Codes Cryptography 86(7) (2018) 1483-1497.


\end{thebibliography}
\end{document}